\documentclass[11pt, a4paper]{article}
\usepackage{amsfonts}
\usepackage{mathrsfs}

\usepackage{latexsym}
\usepackage{graphicx}

\long\def\delete#1{}

\newtheorem{theorem}{Theorem}
\newtheorem{lemma}[theorem]{Lemma}

\usepackage{amssymb}

\def\mod{{\rm mod}}
\def\gcd{{\rm gcd}}

\title{{\bf Identifying codes and locating-dominating sets on paths and cycles}
\thanks{Supported in part by National Natural
Science Foundation of China (Nos. 60673048 and 10871166) and
Shanghai Leading Academic Discipline Project (No. B407).}}

\author{{\bf  Chunxia Chen }
\  \  {\bf Changhong Lu}\footnote{Correspond author. E-mail:
chlu@math.ecnu.edu.cn}\ \\
       Department of Mathematics\\ East China Normal University, Shanghai, 200241,
       China
     \\  \\ {\bf Zhengke Miao}\\
       Department of Mathematics\\
       Xuzhou Normal University, Xuzhou, 221116, China
       }
\begin{document}

\openup 1.0\jot \maketitle

\begin{quote} {\bf Abstract}
Let $G=(V,E)$ be a graph and let $r\ge 1$ be an integer. For  a set
$D \subseteq V$, define $N_r[x] = \{y \in V: d(x, y) \leq r\}$ and
$D_r(x) = N_r[x] \cap D$, where  $d(x,y)$ denotes the number of
edges in any shortest path between $x$ and $y$. $D$ is known as an
$r$-identifying code ~($r$-locating-dominating set, respectively),
if for all vertices $x\in V$~($x \in V\backslash D$, respectively),
$D_r(x)$ are all nonempty and different. 
In this paper, we provide complete results for $r$-identifying codes
in paths and odd cycles; we also give complete results for
$2$-locating-dominating sets in cycles.
\end{quote}
\newpage
\section{Introduction}
We investigate the well-known identifying codes problem which
originated, for instance, from fault diagnosis in multiprocessor
systems.  The purpose of fault diagnosis is to test the
multiprocessor system and locate faulty processors. A multiprocessor
system can be modeled as an undirected graph $G = (V,E)$, where $V$
is the set of processors, $E$ is the set of links in the system.
Specific detectors are executed on certain selected processors to
carry out diagnosis. The selection of these processors is done by
generating the code $D$ that allows for unique identification of
faulty processors. Every processor corresponding to a codeword
vertex tests itself and the processors that are in some areas.
Hence, an optimal code(minimum number of codewords) minimizes the
amount of overhead required to implement fault diagnosis.

More precisely, let $G=(V,E)$ be an undirected graph and let $r\ge
1$ be an integer. Assume that $D$ is a subset of $V$ at which we
place detectors. We define $N_r[x] = \{y \in V: d(x, y) \leq r\}$
and $D_r(x) = N_r[x] \cap D$, where  $d(x,y)$ denotes the number of
edges in any shortest path between $x$ and $y$.  In this sense,
$D_r[x]$ is the set of all detectors that can detect an attack at
$x$.   We say that $D$ is an {\sl $r$-identifying code ($r$-IC) in
$G$ } if $D_r(x) \neq \emptyset$ for every vertex $x\in V$ and
$D_r(x) \neq D_r(y)$ whenever $x\neq y$.   In an $r$-IC, the set of
detectors activated by an attack provides a unique signature that
allows us to determine where the attack took place. We denote the
minimum cardinality of an $r$-identifying code $D$ of $G$ by $M_r^I(G)$. 
Note that not all graphs admit an $r$-identifying code. A necessary
and sufficient condition to admit an $r$-identifying code is that
for  any pair of distinct vertices $x$ and $y$ we have $N_r[x]\neq
N_r[y]$~(\cite{BCO}).

We define a closely related concept as follows. If for all vertices
$x \in V\backslash D$, $D_r(x)$ are all not empty and different,
then we say $D$ is an {\sl $r$-locating-dominating set} or {\sl
$r$-LD set} for short. The smallest $d$ such that there is an $r$-LD
set of size $d$ is denoted by $M_r^{LD}(G)$. The concept was
introduces (for $r = 1$) by Slater \cite{PJS}, motivated by nuclear
power plant safety. It can be used for fault detection in
distributed system. We also note that $r$-LD sets always exist,
since the entire vertex set of a graph is an $r$-LD set.

Identifying codes were introduced in \cite{KCL}, locating-dominating
sets in \cite{CS}. The literature about $r$-identifying codes and
$r$-locating-dominating sets has become quite extensive. There are
now numerous papers dealing with identifying codes and
local-dominating sets (see for instance  \cite{LOB} for an
up-to-date bibliography). The problems  of finding optimal $r$-ICs
or $r$-LDs in a graph are {\em  NP}-complete (see \cite{CH,
COO,COOO,CS}). On the other hand, many special graphs have been
investigated (see for instance,
\cite{BH,BHL,C,CHH,CHL,CHLZ,LR,HL,S}). In this paper we are
interested in studying $r$-IC  for cycles and paths and and $2$-LD
sets for cycles. This subject was already investigated in
\cite{BCO,GMS,RR,PJS,XKH}. Let $P_n$ ($C_n$, respectively) be a path
(cycle, respectively) of $n$ vertices. For $r=1$, the exact values
of $M_1^I(P_n)$ and $M_1^I(C_n)$ for even cycles  was given by
\cite{BCO}; Gravier et al. \cite{GMS} gave the exact values of
$M_1^I(C_n)$ for odd cycles. Its analogue for $1$-LD sets was given
by Slater \cite{PJS}. For $r=2$,  complete solution for $M_2^I(C_n)$
and $M_2^I(P_n)$ was provided in \cite{RR}. Bertand et al.
\cite{BCO} provided complete results about $M_2^{LD}(P_n)$ and gave
the exact value of $M_2^{LD}(C_n)$ for $n=6k$. For all $r\ge 1$,
complete results about $r$-ICs  for even cycles are provided in
\cite{BCO}; Partial results about $r$-ICs for paths and odd cycles
can be found  in \cite{BCO,GMS,XKH}. 

The structure of this paper is the following. Motivated by the
method in \cite{RR}, in Section $2$ we give all values of
$M_r^I(C_n)$ for odd cycle $C_n$.
In Section $3$ we provide complete results for $r$-identifying codes
in paths. In Section $4$  we find the values of $M_2^{LD}(C_n)$.

\section{$r$-identifying codes for  odd cycle $C_n$ }
In the following, we assume that the vertices of  $C_n$ have labeled
consecutively as $x_1,x_2,\cdots,x_n$. When we are dealing with a
cycle, we also use addition and subtraction modulo $n$, so that, for
example, $x_{5n+2}$ means $x_2$.   The case $n$ even is solved in
\cite{BCO}, and it is obvious that $M_r^I(C_n)$ is undefined for
$n\le 2r+1$. Hence,  we assume that  $n\ge 2r+3$. 

\begin{lemma}\label{lem1}
Suppose graph $G$ has maximum degree $2$, $y_1,y_2,\cdots,y_{2r+2}$
is a path in $G$, and $D$ is an $r$-IC for $G$. Then it is
impossible to have $y_1\not\in D$ and $y_{2r+2}\not\in D$.
\end{lemma}

\begin{proof}
If $y_1\not\in D$ and $y_{2r+2}\not\in D$, then
$D_r(y_{r+1})=D_r(y_{r+2})$. \hspace*{\fill}$\Box$

\end{proof}

\begin{lemma}\label{lem2}
If $n\ge 4r+2$, $D$ is an $r$-IC for  $C_n$ if and only if \\
(1)~$x_i\in D$ or $x_{i+2r+1}\in D$ for all $i\in \{1,\cdots,n\}$;\\ 
(2)~there are no $2r+1$ consecutive vertices none of which is in
$D$.
\end{lemma}
\begin{proof}
($\Rightarrow$:)~As $D_r(x) \neq \emptyset$ for all $x\in V$,
condition (2) is necessary. Necessity of condition (1) follows from
Lemma \ref{lem1}.

($\Leftarrow$:)~
Condition $2$ implies that $D_r(x)\neq \emptyset$ for each $x\in V$.
Now we show that  $D_r(x_i) \neq D_r(x_j)$ whenever $i \neq j$.
Without loss of generality, we can assume that  (A) $i < j$ and  the
distance from $x_i$ to $x_j$ is no larger in a clockwise direction
around the cycle than in a counterclockwise direction. $x_{i-r}\neq
x_{i+r+1}$ since $n\ge 4r+2$ (In fact, $n\ge 2r+2$ is sufficient).
By condition (1), either $x_{i-r}\in D$ or $x_{i+r+1}\in D$. If $i+1
\leq j \leq i+2r+1$, this implies that $D_r(x_i) \neq D_r(x_j)$. If
$j
> i+2r+1$, apply condition (2) and (A), we know that
$\{x_{i-r},x_{i-r+1},\cdots,x_i,\cdots,x_{i+r}\}\cap D\neq
\emptyset$, and hence $D_r(x_i) \neq D_r(x_j)$.
\hspace*{\fill}$\Box$
\end{proof}


\begin{theorem}\label{thm3}
For the cycle $C_{2k+1}$, let $k=(2r+1)p+q$ with $p\ge 1$ and $q\in
\{ 0,1, \cdots,2r\}$.\\
(a) ~$M_r^I(C_{2k+1})=\gcd(2r+1,2k+1)\lceil \frac{2k+1}{2\gcd(2r+1,2k+1)}\rceil$ if $\gcd(2r+1,2k+1)\neq 1$;\\
(b)~$M_r^I(C_{2k+1}) = k+2$ if $\gcd(2r+1,2k+1)= 1$ and $q = 0 $, $2r$;\\
(c)~$M_r^I(C_{2k+1})=k+1$ if otherwise.
\end{theorem}
\begin{proof}
By Lemma \ref{lem2}, we have the constraint $x_i\in D$ or $
x_{i+2r+1}\in D$ for $i = 1,\cdots,2k+1$, then there are $2k+1$ such
constraints and each $x_i$ is a term in exactly two of them. Thus,
if $D$ has $d$ vertices, at most $2d$ such constrains are satisfied.
It follows that $D$  must have at least $k+1$ vertices, i.e.,
$|D|\ge k+1$.

For notational convenience, we abbreviate $x_i$ by $i$ and $x_i\in
D$ or $x_j\in D$ by $i\vee j$ in the constraints. Choose $i\in
\{1,\cdots,2r+1\}$, consider the following stream of constraints,
which we call {\sl stream $i$}:
 $$i \vee i + (2r+1) \vee i +
2(2r+1) \vee \cdot\cdot\cdot \vee i + g_i(2r+1) \vee h_i,
$$
 where $i + g_i(2r+1) \le 2k+1 < i + (g_i+1)(2r+1)
\equiv h_i ~(\mod (2k+1))$ and $h_i \in \{1, 2,\cdots, 2r+1\}$.

 Suppose $k = (2r+1)p + q$, where $q \in \{0, 1, \cdots, 2r\}$. 
 It is easy to know that $h_1 = 1$ if and only if $k = (2r+1)p + r$.
Then  stream $1$ leads into stream $h_1$, which leads into stream
$h_{h_1}$, and so on, end with last $h_i=1$. For example, $k=10$ and
$r = 2$, stream $1$: $1\vee 6\vee 11 \vee 16 \vee 21 \vee 5$ which
leads into stream $5$: $5\vee 10\vee 15\vee20\vee 4$, it leads into
stream $4$, and so on, end with stream $1$. When $\gcd
(2r+1,2k+1)=1$, putting the stream together in this order gives us
all constraints, and we call it {\sl full constraint stream}. When
$\gcd(2r+1,2k+1)\neq 1$, it gives us a full stream, we call full
stream $1$, which contains $\frac{2k+1}{\gcd(2r+1,2k+1)}$
constraints and $\frac{2k+1}{\gcd(2r+1,2k+1)} +1$ vertices (where
two $1s$), similarly, we can get full stream $2$, full stream $3$,
$\cdots$, full stream $\gcd(2r+1,2k+1)$. For example, if $k=10$ and
$r = 4$, $\gcd(2r+1,2k+1) = 3$, full stream $3$ is given by $3\vee
12\vee 21\vee 9\vee 18\vee 6\vee 15\vee 3$.

Now we consider the case $\gcd(2r+1,2k+1)\neq 1$. In this case,
there are $\gcd(2r+1,n)$ full streams and each full stream has
$\frac{2k+1}{\gcd(2r+1,2k+1)}$ constraints, thus we need at least
$\lceil\frac{2k+1}{2\gcd(2r+1,2k+1)}\rceil$ vertices from each to be
put into $D$, which follows that $|D| \geq
\gcd(2r+1,2k+1)\lceil\frac{2k+1}{2\gcd(2r+1,2k+1)}\rceil$. Note that
$\frac{2k+1}{\gcd(2r+1,2k+1)}$ is odd, hence each full stream has
even vertices. Let us denote $f_i^j$ the $j$-th vertex of full
stream $i$ for $1\le i\le \gcd(2r+1,2k+1)$ and $1\le j\le
\frac{2k+1}{\gcd(2r+1,2k+1)}+1$. Let $D = \{f_i^j: $ $1\le i\le
\gcd(2r+1,2k+1)$, $1\le j\le \frac{2k+1}{\gcd(2r+1,2k+1)}+1$, $i+j$
is odd $\}$.
It is easy to check that $D$ satisfies the conditions in Lemma
\ref{lem2}, and hence $D$ is an $r$-IC of $C_{2k+1}$ with
$\gcd(2r+1,2k+1)\lceil\frac{2k+1}{2\gcd(2r+1,2k+1)}\rceil$ vertices.
The proof of (a) is complete. In the following, we assume that
$\gcd(2r+1,2k+1)=1$.




We first consider the case $q=2r$. In this case,  the full
constraint stream consists of stream $1$, stream $2$, $\cdots$,
stream $2r+1$ in turn.
Suppose $D$ is an $r$-IC with $|D|=k+1$. Since the full constraint
stream has $2k+1$ constraints and each vertex is in exactly two
constraints, then if we can satisfy all the constraints with $k+1$
vertices of $D$, there must be exactly one constraint where both
vertices are in $D$, and all other constraints have exactly one of
their vertices in $D$. Without loss of generality, we assume that
$1\in D$ and  $1 + (2r+1)\in D$. Then, the rest of the membership of
$D$ is forced upon us:\\
\newline
$\bullet$ from stream $1$: use vertices $1$ and $1+z(2r+1)$, $z$ is odd;\\
$\bullet$ from stream $i$: use  vertices $i+z(2r+1)$, $z$ is odd, for $i=2,3,\cdots,2r+1$.\\
\newline
This satisfies condition (1) of Lemma \ref{lem2} and uses $k+1$
vertices. However, when $p\ge 1$, condition (2) of Lemma \ref{lem2}
is violated, since then $1+2(2r+1),2+2(2r+1),\cdots,(2r+1)+2(2r+1)$
are
not in $D$. So, we conclude that $M_r^I(C_{2k+1}) \geq k+2$. 
 Now we construct an $r$-IC with $k+2$ vertices as follows:\\
\newline
$\bullet$ stream $i$: select vertices $i+z(2r+1)$, $z$ is odd, $i \neq r+1$;\\
$\bullet$ stream $r+1$: select vertices $r+1+z(2r+1)$, $z$ is even;\\
 $\bullet$ add  vertices $1$ and $r+1+(2p+1)(2r+1)$.
\newline

We now turn to the case $q=0$. In this case,  $k=(2r+1)p$ and $p\ge
1$. The full constraint stream consists of stream $2r+1$, stream
$2r$, $\cdots$, stream $1$ in turn. Similar to the case $q=2r$, we
can prove that  $M_r^I(C_{2k+1}) \geq k+2$. 
We construct an $r$-IC $D$ with $k+2$ vertices as follows:\\
\newline
$\bullet$ stream $i$: select vertices $i+z(2r+1)$, $z$ is odd, $i = 1,2r+1$;\\
$\bullet$ stream $j$: select vertices $j+z(2r+1)$, $z$ is even, $j = 2,3,\cdots,2r$.\\
$\bullet$  add  vertices $1$ and $1 + 2p(2r+1)$.\\

We now turn to the case $1\le q\le r-1$.  In this case, 
we just need to construct an $r$-IC $D$ with $|D|= k+1$. As
discussion  above, if we select  $1$ and $1+(2r+1)$ into $D$, then
the other vertices of $D$ are fixed. Now we prove that $D$ satisfies
two conditions of Lemma \ref{lem2}, and hence $D$ is an $r$-IC.
Obviously,  it satisfies condition (1) of Lemma \ref{lem2}. We just
need to show it satisfies condition (2) of Lemma \ref{lem2}, i.e.
there are no $2r+1$ consecutive vertices none of which is in $D$.
Suppose to the contrary that  $i + g(2r+1), i+1 + g(2r+1), \cdots,
2r+1 + g(2r+1), 1 + (g+1)(2r+1), \cdots, i-1 + (g+1)(2r+1)\not\in D$
for some $i\in \{1,2,\cdots,2r+1\}$ and $g\in \{0,1,\cdots,2p\}$. If
$i = 1$, then $1 + g(2r+1), 2 + g(2r+1), \cdots, 2r+1 + g(2r+1)
\not\in D$. Since $1, 1+(2r+1)$ are selected into  $D$, then
$1+z(2r+1)\in D$ for all odd  $z\le 2p$. So $g$ must be even. On the
other hand, stream $1: 1\vee 1+(2r+1)\vee \cdots \vee 1+2p(2r+1)\vee
2r-2q+1$,  leads into stream $2r-2q+1$, and $2r-2q+1+z(2r+1)\in D$
for all even  $z\le 2p$, it contradicts that $g$ is even and
$2r-2q+1+g(2r+1)\not\in D$. So, $i>1$. Since $1+(g+1)(2r+1)\not\in
D$,  hence $g$ is odd.

As $2r-2q+1+z(2r+1)\in D$ for all even  $z\le 2p$,  we know that
$2r-2q+1\geq i$. If $2r-2q+1 \leq 2q+1$, then  stream $2r-2q+1$
leads into stream $4r-4q+1$ and $4r-4q+1+z(2r+1)\in D$ for all odd
$z\le 2p$, thus, $4r-4q+1<i$. It contradicts that $2r-2q+1\geq i$.
So, $2r-2q+1> 2q+1$. Then stream $2r-2q+1$ leads into stream $2r-4q$
and $2r-4q+z(2r+1)\in D$ for all even $z\le 2p$. Therefore, $2r-4q
\geq i$. Similarly, we can show that $2r-4q>2q+1$ and stream $2r-4q$
leads into stream $2r-6q-1$. Let $m_0$ be the minimum integer such
that $2+2r-m_0(2q+1)\le max\{i,2q+1\}$.
We have that: stream  $1 \Rightarrow$ stream  $2r+2-(2q+1)
\Rightarrow$ stream  $2r+2-2(2q+1)\Rightarrow \cdots \Rightarrow$
stream $2r+2-(m_0-1)(2q+1) \Rightarrow$ stream $2r+2-m_0(2q+1)$. By
the selection of $m_0$, we know that $2r+2-m(2q+1)+z(2r+1)\in D$ for
all $1\le m\le m_0$ and all even $z\le 2p$. Hence,
$2r+2-m_0(2q+1)\ge i$ and it implies that $2r+2-m_0(2q+1)\le 2q+1$.
Therefore, stream $2r+2-m_0(2q+1)$ leads into stream
$4r-2q+2-m_0(2q+1)$ and $4r-2q+2-m_0(2q+1)+z(2r+1)\in D$ for all odd
$z\le 2p$. So, we know that $4r-2q+2-m_0(2q+1)<i$. It contradicts
that $2r+2-m_0(2q+1)\ge i$.
So, there are no $2r+1$ consecutive vertices none of which is in
$D$, and hence  $D$ is an $r$-IC. Therefore, $M_r^I(C_{2k+1}) = k+1$
in this case.

Note that $\gcd(2k+1,2r+1)=1$ implies that $q\neq r$. At last, we
consider the case $r+1\le q\le 2r-1$. In this case,
$2k+1=(2p+1)(2r+1)+2(q-r)$. We just need to construct an $r$-IC $D$
with $|D|= k+1$. As discussion above, if we select  $2r+1$ and
$2r+1+(2r+1)$ into $D$, then the other vertices of $D$ are fixed.
Obviously,  $D$ satisfies condition (1) of Lemma \ref{lem2}. We just
need to show that there are no $2r+1$ consecutive vertices none of
which is in $D$. Suppose to the contrary that  $i + g(2r+1), i+1 +
g(2r+1), \cdots, 2r+1 + g(2r+1), 1 + (g+1)(2r+1), \cdots, i-1 +
(g+1)(2r+1)\not\in D$ for some $i\in \{1,2,\cdots,2r+1\}$ and $g\in
\{0,1,\cdots,2p+1\}$. Since $2r+1$ and $2r+1+(2r+1)$ are selected
into $D$, $2r+1+z(2r+1)\in D$ for all odd $z\le 2p$. Hence, $g$ is
even.

If $i = 1$, then $1 + g(2r+1), 2 + g(2r+1), \cdots, 2r+1 + g(2r+1)
\not\in D$. Consider stream $2r+1: 2r+1\vee 2r+1+(2r+1)\vee \cdots
\vee 2r+1+2p(2r+1)\vee 4r-2q+1$, which leads into stream $4r-2q+1$,
and $4r-2q+1+z(2r+1)\in D$ for all even  $z\le 2p+1$, it contradicts
that $g$ is even and $4r-2q+1+g(2r+1)\not\in D$. Hence $i>1$.

As $4r-2q+1+z(2r+1)\in D$ for all even  $z\le 2p+1$,  we know that
$4r-2q+1<i$. If $4r-2q+1 > 2(q-r)$, then  stream $4r-2q+1$ leads
into stream $6r-4q+1$ and $6r-4q+1+z(2r+1)\in D$ for all odd $z\le
2p+1$, thus, $6r-4q+1\ge i$. It contradicts that $4r-2q+1< i$. So,
$4r-2q+1 \le 2(q-r)$, then stream $4r-2q+1$ leads into stream
$8r-4q+2$ and $8r-4q+2+z(2r+1)\in D$ for all even $z\le 2p+1$.
Therefore, $8r-4q+2 < i$. Similarly, we can show that $8r-4q+2\le
2(q-r)$ and stream $8r-4q+2$ leads into stream $12r-6q+3$. The
following proof is similar to that of the case $1\le q\le r-1$. The
detail is left to readers. \hspace*{\fill}$\Box$

\begin{lemma}\label{lem19}
If $3r+2\le n\le 4r+1$, $D$ is a $r$-IC for  $C_n$ if and only if
$x_i\in D$ or $x_{i+2r+1}\in D$ for all $i\in \{1,\cdots,n\}$.
\end{lemma}
\begin{proof}
 Necessity  follows from Lemma \ref{lem1}. We shall now observe
 sufficiency. Assume that $D_r(x_i)= \emptyset$ for some $x_i\in V$.
 Then, none of vertices in $\{x_{i-r},x_{i-r+1},\cdots,x_{i+r}\}$ is in
$D$. As $x_i\in D$ or $x_{i+2r+1}\in D$ for all $i\in
\{1,\cdots,n\}$, it follows that these vertices
$x_{i+r+1},x_{i+r+2},\cdots,x_{i+3r+2}$ are all contained in $D$.
This contradicts that $n\le 4r+1$. Hence, $D_r(x_i)\neq \emptyset$
for all $x_i\in C_n$. The  proof of $D_r(x_i)\neq D_r(x_j)$ ( $i\neq
j$)  is same as that in Lemma \ref{lem2}, and hence omitted in here.
\hspace*{\fill}$\Box$
\end{proof}

Using Lemma \ref{lem19}, we can  prove the following theorem, which
has been shown by Gravier et al. (cf. Theorem 7 in \cite{GMS}). The
proof is similar to partial proof  of Theorem \ref{thm3}, and
omitted in here.
\begin{theorem}\label{thm20}(cf. Theorem 7 in
\cite{GMS}) $M_r^I(C_{2k+1})
=\gcd(2r+1,2k+1)\lceil\frac{2k+1}{2\gcd(2r+1,2k+1)}\rceil$ for
$3r+2\le 2k+1\le 4r+1$.
 \end{theorem}


\end{proof}
\vskip 0.2cm

Next, we discuss the value of $M_r^I(C_n)$ for $2r+5 \le n\le 3r+1$,
which is the remainder case in \cite{BCO,GMS,XKH}.  A lemma is given
firstly as follows.

\begin{lemma}\label{lem5}
If $2r+5 \leq n \leq 3r+1$, let $n = 2r+1+q$ ($4 \leq q \leq r$),
$D$ is an $r$-IC for $C_n$ if and only if \\
(1)~ $x_i\in D$ or $x_{i+q}\in D$ for all $i\in \{1,\cdots,n\}$; \\
(2)~ there is at most one set $\{x_{i+1}, x_{i+2},\cdots, x_{i+q}\}$
such that none of which is   in $D$.
\end{lemma}


\begin{proof}
($\Rightarrow$:)~Suppose to the contrary that $x_i \not\in D$ and
$x_{i+q}\not\in D$. Since  $N_r[x_{i-r-1}] = V\backslash \{x_i,
x_{i+1},\cdots, x_{i+q-1}\}$ and  $N_r[x_{i-r}] = V\backslash
\{x_{i+1}, x_{i+2},\cdots, x_{i+q}\}$, this leads to the equality
$D_r(x_{i-r-1}) = D_r(x_{i-r})$, a contradiction. If there exist two
distinct sets $\{x_{i+1}, x_{i+2},\cdots, x_{i+q}\}$ and $\{x_{j+1},
x_{i+2},\cdots, x_{j+q}\}$ such that none of which is in $D$, then
$D_r(x_{i-r}) = D_r(x_{j-r})$, which follows from $N_r[x_{i-r}] =
V\backslash\{x_{i+1}, x_{i+2},\cdots, x_{i+q}\}$ and $N_r[x_{j-r}] =
V\backslash \{x_{j+1}, x_{j+2},\cdots, x_{j+q}\}$.

($\Leftarrow$:)~For every vertex $x_i\in V$, $N_r[x_i] = \{x_{i-r},
x_{i-r+1},\cdots, x_{i+r}\}$.  As $q \leq r$, both $x_i$ and
$x_{i+q}$ are  in $N_r[x_i]$, and by condition (1),  we can conclude
that $D_r(x_i) \neq \emptyset$. For distinct vertices $x_i$ and
$x_j$, without loss of generality, we assume that $i < j$ and the
distance from $x_i$ to $x_j$  in a clockwise direction around the
cycle is no larger than in a counterclockwise direction. If $i+1
\leq j \leq i+q$,  we have $x_{i+r+1} \in N_r[x_j]\backslash
N_r[x_i]$ and $x_{i+r+q+1} \in N_r[x_i]\backslash N_r[x_j]$. Hence,
by condition (1), $D_r(x_i) \neq D_r(x_j)$. Assume that  $j > i+q$.
Let $A = \{x_{i+r+1}, x_{i+r+2}, \cdots, x_{i+r+q}\}$ and $B =
\{x_{j+r+1}, x_{j+r+2},\cdots, x_{j+r+q}\}$. Then  $N_r[x_{i}] =
V\backslash A$ and $N_r[x_{j}] = V\backslash B$. Since $j>i+q$ and
the distance from $x_i$ to $x_j$  in a clockwise direction around
the cycle is no larger than in a counterclockwise direction, it
implies that  $A \cap B = \emptyset$.  By condition (2), either
$A\cap D\neq \emptyset$ or $B\cap D\neq \emptyset$ holds. Without
loss of generality, we assume that $A\cap D \neq \emptyset$ and
$x_{i+r+t} \in D$ for some $t\in \{1,2,\cdots,q\}$. Then
 $x_{i+r+t} \not\in D_r(x_i)$, but  $x_{i+r+t} \in D_r(x_j)$ as
 $A\cap B=\emptyset$.
Hence, we have that $D_r(x_i) \neq D_r(x_j)$. \hspace*{\fill}$\Box$
\end{proof}

\begin{theorem}\label{thm6}
For the cycle $C_{2k+1}$ with $2r+5\le 2k+1\le 3r+1$, let $2k+1 =
2r+1+q = lq + m$, where $l \geq 3 $ is an integer and  $m \in
\{0, 1,\cdots, q-1\}$, 
then\\
(1)~ $M_r^I(C_{2k+1})=k+2$ if $l$ is odd, $m=q-1$,  $2k+1\ge 5q$ or
$l$ is even, $m=1$;\\
(2)~ $M_r^I(C_{2k+1})=\gcd (q, 2k+1)\lceil\frac{2k+1}{2gcd(q,
2k+1)}\rceil$  if  otherwise.

\end{theorem}

\begin{proof}
Since $2k+1= 2r+1+q = lq + m$, then $q$ is even, and hence $m$ is
odd.  Let $D$ be an $r$-IC for $C_{2k+1}$, by Lemma \ref{lem5},  it
must satisfy $2k+1$ constraints: $x_i \in D$ or $x_{i+q} \in D$ for
$i = 1, 2,\cdots, 2k+1$. Similarly, we abbreviate $x_i$ by $i$ and
$x_i\in D$ or $x_j\in D$ by $i\vee j$ in the constraints. For $i\in
\{1,2,\cdots,q\}$, we define  stream $i$ as follows:
 $$i \vee i + q \vee i +
2q \vee \cdot\cdot\cdot \vee i + g_iq \vee h_i,
$$
 where $i + g_iq \le 2k+1 < i + (g_i+1)q
\equiv h_i ~(\mod (2k+1))$ and $h_i \in \{1, 2,\cdots, q\}$.

Then stream $1$ leads into stream $h_1$, which leads into stream
$h_{h_1}$, and so on, end with last $h_i=1$. When $gcd(q,2k+1)\neq
1$, it gives us a full stream, denoted by {\sl full stream $1$},
which contains $\frac{2k+1}{gcd(q,2k+1)}$ constrains and
$\frac{2k+1}{\gcd(q,2k+1)}+1$ vertices (where two $1$s). Similarly,
we can get {\sl full stream $2$}, $\cdots$, {\sl full stream
$gcd(q,2k+1)$}. There are $\gcd(q,2k+1)$ full streams and each full
stream contains $\frac{2k+1}{\gcd(q,2k+1)}$ constrains and
$\frac{2k+1}{\gcd(q,2k+1)}+1$ vertices (where two $i$s for $1\le
i\le \gcd(q,2k+1)$). To satisfy all constrains, we need  at least
$\frac{2k+1}{2\gcd(q,2k+1)}$ vertices from each full stream to be
put into $D$. This follows that $|D|\ge \gcd(q,
2k+1)\lceil\frac{2k+1}{2\gcd(q, 2k+1)}\rceil$.

Let us denote $f^j_i$ the $j$-th vertex of full stream $i$ for $1\le
i\le \gcd(q,2k+1)$ and $1\le j\le \frac{2k+1}{\gcd(q,2k+1)}+1$.
Since $2k+1$ is odd, thus each full stream has even vertices. Let
$D=\{f^j_i:$ $1\le i\le \gcd(q,2k+1)$, $1\le j\le
\frac{2k+1}{\gcd(q,2k+1)}+1$, $i+j$ is odd $\}$. It is easy to see
that $D$ satisfies conditions of Lemma \ref{lem5}. Hence, $D$ is an
$r$-IC
with $|D|=\gcd(q, 2k+1)\lceil\frac{2k+1}{2\gcd(q, 2k+1)}\rceil$. 

We now turn to the case $ \gcd(q,2k+1)=1$.  $\gcd(q, 2k+1) = 1$
implies that there is only one full stream, which contains  $2k+1$
constraints. We discuss it as the following cases.

Case~1. $m=q-1$, $2k+1=lq+q-1$. 

 In this case, the full stream consists of stream $1$, stream $2$, $\cdots$, stream $q$ in turn.
 Suppose that  $D$ is  an $r$-IC for $C_{2k+1}$ with  $|D|=k+1$. Since there are
$2k+1$ constraints,  there must be exactly one constraint where both
vertices are in $D$, and all other constraints have exactly one of
their vertices in $D$. Without loss of generality, we take $1$ and
$1+q$ in $D$, then the rest of the membership of $D$ is  forced upon
us. If $l$ is even,  the membership of $D$ is just  the following
vertices:\\
\newline
$\bullet$ from stream $i$: use vertices $i+zq$, where the parity of
$i$ and $z$ is the same, for $i=1,2,\cdots, q$.\\

It is easy to check that there are no $q$ consecutive vertices none
of which is in $D$.  By  Lemma \ref{lem5}, $D$ is an $r$-IC with
$|D|=k+1$.

If $l$ is odd,  the membership of $D$ is just  the following
vertices:\\
\newline
$\bullet$ from stream $i$: use  vertices $i+zq$, $z$ is odd, $i=1,2,3,\cdots,q$.\\

If $l$ is odd and $2k+1<5q$, i.e., $2k+1=4q-1$, then $D$ satisfies
conditions of lemma \ref{lem5}. Hence, $D$ is also an $r$-IC  with
$|D|=k+1$. However, when  $l$ is odd and $2k+1\ge 5q$, condition (2)
of Lemma \ref{lem5} is violated, there exist  two sets  $\{1+2q,
2+2q,\cdots, q+2q\}$ and $\{1+4q, 2+4q,\cdots, q+4q\}$ such that
none of which is in $D$. Then $D$ is not an $r$-IC. So, we conclude
that $M_r^I(C_{2k+1}) \geq k+2$. Now we construct an $r$ - IC with $k+2$ vertices as follows:\\
\newline
$\bullet$ from stream~$i:$ use vertices $i+zq$, $z$ is odd and $i
\neq\frac{q}{2}+1$;\\
$\bullet$ from stream $\frac{q}{2}+1:$ use vertices $\frac{q}{2}+1+zq$, $z$ is even; \\
$\bullet$ add the vertex $\frac{q}{2}+1+lq$.\\

Case~2.~$m = 1, 2k+1 = lq+1$.

In this case, the full stream consists of stream $q$, stream $q-1$,
$\cdots$, stream $1$ in turn. We can also  prove that
$M_r^I(C_{2k+1})=k+1$ if $l$ is odd and $M_r^I(C_{2k+1})\ge k+2$ if
$l$ is even. The proof is analogous with case 1 and  is omitted in
here. If $l$ is even, we can construct an $r$-IC with $k+2$
vertices as follows:\\
\newline
$\bullet$ from stream~$i:$ use vertices $i+zq$, $z$ is odd and $i
\neq\frac{q}{2}+1$;\\
$\bullet$ from stream $\frac{q}{2}+1:$ use vertices $\frac{q}{2}+1+zq$, $z$ is even. \\

Case~3.~$1 < m < q -1, 2k+1 = lq + m$

 Let $D$ denote an $r$-IC for $C_{2k+1}$ with $k+1$ vertices. When $l$ is odd, without loss of generality, we
take $q$ and $2q$ in $D$, the rest of the membership of $D$ is
forced upon us, and condition (1) of Lemma \ref{lem5} holds. Next,
we prove that  there are no $q$ consecutive vertices none of which
is in $D$. i.e., $D$ satisfies a stronger property  than condition
(2) of Lemma \ref{lem5}.  Suppose to the contrary that  $i + pq,
i+1+pq,\cdots, q+pq, 1+(p+1)q, 2+(p+1)q,\cdots, (i-1)+(p+1)q \not\in
D$ for some $i\in \{1,2,\cdots,q\}$ and $p\in \{0,1,\cdots l\}$. By
the selection of $D$, we know that $q\in D$ and $q+zq\in D$ for all
odd $z\le l $. So, $p$ is an even positive integer. If $i = 1$, it
implies that $1+pq, 2+pq,\cdots, q+pq \not\in D$. Stream $q: q\vee
2q\vee \cdots \vee lq\vee q-m$,  leads into stream $q-m$, and
$q-m+zq\in D$ for all even  $z\le l$, it contradicts that $p$ is
even and $q-m+pq\not\in D$. Hence, $i>1$.

Since $q-m+zq \in D$ for all even  $z\le l$, then $q-m\le i-1$. If
$q-m>m$, then stream $q-m$ leads into  stream $q-2m$, and $q-2m+zq
\in D$ for all odd $z\le l$.  Thus, $q-2m \geq i$. It contradicts
that $q-m\le i-1$. Hence, $q-m \leq m$. Then stream $q-m $ leads
into stream $2q-2m$, and $2q-2m+zq \in D$ for all even $z\le l$.
Therefore, $2q-2m \le i-1$. Similarly, we have $2q-2m\le m$ and
stream $2q-2m$ leads into stream $3q-3m$. Let $t_0$ be the minimum
integer such that $t_0(q-m)> max\{i-1,m\}$. We have that: stream $q
\Rightarrow$ stream  $q-m \Rightarrow$ stream  $2(q-m)) \Rightarrow$
$ \cdots \Rightarrow$ stream $(t_0-1)(q-m) \Rightarrow$ stream
$t_0(q-m)$. By the selection of $t_0$, we know that $t(q-m)+zq\in D$
for all $1\le t\le t_0$ and all even $z\le l$. Hence, $t_0(q-m)\le
i-1$ and it implies that $t_0(q-m)> m$. Therefore, stream $t_0(q-m)$
leads into stream $t_0(q-m)-m$ and $t_0(q-m)-m+zq\in D$ for all odd
$z\le l$. So, we know that $t_0(q-m)-m\ge i$. It contradicts that
$t_0(q-m)\le i-1$.
 So, there are no $q$ consecutive vertices none of which is in
$D$, and hence  $D$ is an $r$-IC with $k+1$ vertices.

When $l$ is even, without loss of generality, we take $1$ and $1+q$
in $D$, the rest of the membership of $D$ is forced upon us. The
remainder  proof is analogous and is omitted in here.
$\hfill\Box$\\
\end{proof}

The value of $M_r^I(C_{2r+3})$ was first obtained in \cite{GMS}. We
present them here for completeness.

\begin{theorem}\label{thm21}(cf Theorem 5 in \cite{GMS})~
 $M_r^I(C_{2r+3}) =\lfloor\frac{4r+6}{3}\rfloor$ for all $r\ge 1$.
\end{theorem}
\begin{proof}
Let $D$ be an $r$-IC of $C_{2r+3}$. For all $i$ we have
$D_r^I(x_{i-r-1})=D\backslash \{x_i,x_{i+1}\}$. Thus, the two
following assertions are true:\\
(a) there is at most a pair $\{x_i,x_i+1\}$ such that $x_i\not\in D$
and $x_{i+1}\not\in D$;\\
(b) there is no pair $\{x_i,x_{i+2}\}$ such that $x_i\not\in D$ and
$x_{i+2}\not\in D$. (For otherwise,
$D_r^I(x_{i-r-1})=D_r^I(x_{i-r})$).

Let us then partition  the vertices of $C_{2r+3}$ into blocks of
three consecutive vertices, plus possibly one block consisting in
one or two vertices. If there is a pair $\{x_i,x_{i+1}\}$ such that
$x_i\not\in D$ and $x_{i+1}\not\in D$, then we can partition the
vertices such that $x_i$ and $X_{i+1}$ are not in the same block. By
(a) we know that there is at most one such pair. By (b), any
three-element block of the partition contains at least one vertex in
$D$. This leads to the inequality $M_r^I(C_{2r+3}) \ge
\lfloor\frac{4r+6}{3}\rfloor$.

Now we construct an $r$-IC of $C_{2r+3}$ to attaint this bound.\\
\newline
$\bullet$ $2r=0(\mod \ \ 3)$, $D=\{x_i| i=1(\mod \ \ 3)$ or $2(\mod
\ \ 3)\}$;\\
$\bullet$ $2r=1(\mod\ \ 3)$, $D=\{x_i| 1\le i\le 2r+2, i=1(\mod\ \
3)$ or
$2(\mod\ \ 3)\}$;\\
$\bullet$ $2r=2(\mod\ \ 3)$, $D=\{x_i| 1\le i\le 2r+1, i=1(\mod\ \
3)$ or $2(\mod\ \ 3)\}\cup \{x_{2r+3}\}$.
$\hfill\Box$\\
\end{proof}

\section{$r$-identifying codes for path  $P_n$}

We turn now to the path $ P_n$.  We assume that the vertices of
$P_n$ have labeled consecutively as $x_1,x_2,\cdots,x_n$. First it
is easy to know that $M_r^I(P_n)$ is undefined if and only if $n\le
 2r$. In the following, we assume that $n\ge 2r+1$. 

\begin{lemma}\label{lem9}
If $D$ is an $r$-IC for $P_n$, then $x_{r+2}, x_{r+3}, \cdots,
x_{2r+1} \in D$ and $x_{n-r-1}, x_{n-r-2},\cdots , x_{n-2r} \in D$.
\end{lemma}

\begin{proof}
For $i=1,2,\cdots, r$, $D_r(x_i)\neq D_r(x_{i+1})$ implies that
$x_{i+r+1}\in D$, and $D_r(x_{n-i})\neq D_r(x_{n-i+1})$ implies that
$x_{n-r-i}\in D$. $\hfill\Box$
\end{proof}

\begin{lemma}\label{lem10}
If $n\ge 2r+1$, $D$ is an $r$-IC for $P_n$ if and only if  the
following conditions holds:\\
(1) there are no $2r+2$ consecutive vertices with the first and last
not in $D$; \\
(2) there are no $2r+1$ consecutive vertices none of which is in
$D$;\\
(3) $\{x_{r+2}, x_{r+3},  \cdots, x_{2r+1}\}\subseteq D$ and
$\{x_{n-r-1}, x_{n-r-2}, \cdots, x_{n-2r}\} \subseteq D$.\\
(4) $\{x_1, x_2,  \cdots, x_{r+1}\}\cap  D\neq \emptyset$ and
$\{x_n, x_{n-1}, \cdots, x_{n-r}\} \cap D\neq \emptyset$.
\end{lemma}

\begin{proof}
($\Rightarrow:$) Necessity of (1)  follows from Lemma \ref{lem1},
and necessity of (2) follows from $D_r(x)\neq \emptyset$ for every
vertex $x\in V$. Necessity of (3) follows from Lemma \ref{lem9}, and
necessity of (4) follows from  $D_r(x_1)\neq \emptyset$ and
$D_r(x_n)\neq \emptyset$.

 ($\Leftarrow:$)  By conditions (2), (3) and (4), $D_r(x)\neq \emptyset$ for
every vertex $x\in V$. Consider $x_i$ and $x_j$, without loss of
generality, we assume that $i<j$. If $i+1 \leq j \leq i+2r+1$ and
$i>r$,  by condition (1), either $x_{i-r} \in D$ or $x_{i+r+1} \in
D$ holds, and hence $D_r(x_i) \neq D_r(x_j)$. If $i+1 \leq j \leq
i+2r+1$ and $i\le r$, by condition (3), we have $x_{i+r+1} \in D$,
and hence $D_r(x_i) \neq D_r(x_j)$. If $j > i+2r+1$ and  $i > r$, by
condition (2), $\{x_{i-r}, x_{i-r+1}, \cdots, x_{i+r}\}\cap D\neq
\emptyset$, so $D_r(x_i) \neq D_r(x_j)$.  If $j > i+2r+1$ and $i
\leq r$, by condition (4), $\{x_1, x_2, \cdots, x_{r+1}\}\cap D\neq
\emptyset$, so $D_r(x_i) \neq D_r(x_j)$. $\hfill\Box$ \\

\end{proof}

Lemma \ref{lem10} allows us to proceed for a path much as we did
with cycle. Constraint streams are again the focus of our argument.
Similarly, we use $i$ as an abbreviation for vertex $x_i$ and we
modify the definition of constraint stream $i$ to omit the last term
$h_i$. i.e., we define stream $i$  as follows:

 $$i \vee i + (2r+1) \vee i + 2(2r+1) \vee \cdots \vee i +
g_i(2r+1),
$$
 where $i + g_i(2r+1) \le n$ and $1\le i\le 2r+1$.

The following theorem gives all results  for $M_r^I(P_n)$.

\begin{theorem}\label{thm11}
Let $n=(2r+1)p+q$, $p\ge 1$, $q \in \{0,1,2,\cdots,2r\}$.\\
(1) If $q =0$, then $M_r^I(P_n)= \frac{(2r+1)p}{2}+1$ if $p$ is
even; $M_r^I(P_n)= \frac{(2r+1)(p-1)}{2}+2r$ if $p$ is odd.\\
(2) If $1 \leq q \leq r+1$,  then $M_r^I(P_n)= \frac{(2r+1)p}{2}+q$
if $p$ is even; $M_r^I(P_n)= \frac{(2r+1)(p-1)}{2}+2r+1$ if $p$ is
odd.\\
(3) If $r+2 \leq q \leq 2r$,  then $M_r^I(P_n)=
\frac{(2r+1)p}{2}+q-1$ if $p$ is even; $M_r^I(P_n)=
\frac{(2r+1)(p-1)}{2}+2r+1$ if $p$ is odd.
\end{theorem}

\begin{proof}
Let $D$ be an $r$-IC for $P_n$. We first discuss the case $q=0$.

(1) If $q=0$,  then $r+2, r+3, \cdots, 2r+1, 1+(p-1)(2r+1),
2+(p-1)(2r+1), \cdots, r+(p-1)(2r+1)\in D$, which follows from
condition (3) of Lemma \ref{lem10}. For $i \in \{ 1, 2,\cdots,
2r+1\}$, the constraint stream $i$ is given as follows: $i \vee
i+(2r+1) \vee \cdots \vee i+(p-1)(2r+1) $.
To satisfy condition (4) of Lemma \ref{lem10}, there are four possible cases:\\
(1A)~$r+1\in D$ and $r+1+(p-1)(2r+1) \in D$;\\
(1B)~$i\in D$ for some $i\in \{1, 2, \cdots, r\}$ and $r+1+(p-1)(2r+1) \in D$;\\
(1C)~$r+1\in D$ and $j+(p-1)(2r+1) \in D$ for some $j\in\{r+2, r+3, \cdots, 2r+1\}$;\\
(1D)~$i\in D$ for some $i\in \{1, 2, \cdots, r\}$ and  $j+(p-1)(2r+1) \in D$
 for some $ j\in\{r+2, r+3, \cdots, 2r+1\}$.

First consider the case (1A). For each stream $i$ ($i\in
\{1,2,\cdots,r\}$), we have already taken $i+(p-1)(2r+1)$ into $D$,
satisfying the last constrain, and there are $p-2$ remaining
constraints. So, we need take at least
  $\lceil\frac{p-2}{2}\rceil$ vertices from each stream
$i$ ($i\in \{1,2,\cdots,r\}$)  into $D$ to satisfy the remaining
constraint. Turn to stream $r+1$, since $r+1$ and $r+1+(p-1)(2r+1)$
are already put in $D$, satisfying the first and last constraints in
stream $r+1$, so, we need at least $\lceil\frac{p-3}{2}\rceil$ to
satisfy the remaining constraints. Similarly, it requires  at least
$\lceil\frac{p-2}{2}\rceil$ vertices to satisfy the remaining
constraints in each stream $i$ for $i=r+2,r+3,\cdots,2r+1$. Hence,
we need at least
$2r+2+\lceil\frac{p-3}{2}\rceil+2r\lceil\frac{p-2}{2}\rceil$
vertices in all. 

Now we consider the case (1B). For stream $i$, we have already taken
$i$ and $i+(p-1)(2r+1)$ into $D$, satisfying the first and the last
constrains,  then we need at least $\lceil\frac{p-3}{2}\rceil$
vertices to satisfy the remaining constraints. For each of  the
other streams, we need at least $\lceil\frac{p-2}{2}\rceil$ vertices
to satisfy the remaining constraints. Thus, we need at least
$2r+2+\lceil\frac{p-3}{2}\rceil+2r\lceil\frac{p-2}{2}\rceil$
vertices.

We now turn to the case (1C). For stream $j$, we have already taken
$j$ and $j+(p-1)(2r+1)$ into $D$, satisfying the first and the last
constrains,  then we need at least $\lceil\frac{p-3}{2}\rceil$
vertices to satisfy the remaining constraints. For each of  the
other streams, we need at least $\lceil\frac{p-2}{2}\rceil$ vertices
to satisfy the remaining constraints. Thus, we need at least
$2r+2+\lceil\frac{p-3}{2}\rceil+2r\lceil\frac{p-2}{2}\rceil$
vertices.

 At last we consider the case (1D). For stream $i$,  we have already taken $i$ and $i+(p-1)(2r+1)$ into $D$,
satisfying the first and the last constrains,  then we need at least
$\lceil\frac{p-3}{2}\rceil$ vertices to satisfy the remaining
constraints. For stream $j$, similarly, we need at least
$\lceil\frac{p-3}{2}\rceil$ vertices to satisfy the remaining
constraints. For stream $r+1$, we need at leat
$\lceil\frac{p-1}{2}\rceil$ vertices to satisfy its constrains. For
each of the other streams, we need at least
$\lceil\frac{p-2}{2}\rceil$ vertices to satisfy the remaining
constraints. Hence, we need at least
$2r+2+2\lceil\frac{p-3}{2}\rceil+(2r-2)\lceil\frac{p-2}{2}\rceil+\lceil\frac{p-1}{2}\rceil$
vertices.

 Finally, comparing the required minimum number of $D$ in
all four cases, we see that when $p$ is even, the minimum is
$\frac{(2r+1)p}{2}+1$, which is achieved in both cases (1A), (1B)
and (1C) and when $p$ is odd, the minimum is
$\frac{(2r+1)(p-1)}{2}+2r$, which is achieved in case (1D).

Next, we construct an $r$-IC, achieving the bound as follows:\\
 When $p$ is even, \\
 \newline
 $\bullet$ from stream~$i$: use vertices $i + z(2r+1)$, where $z$ is even, for
$i = 1, r+2, r+3, \cdots, 2r+1$;\\
$\bullet$  from stream~$j$: use vertices $j + z(2r+1)$,~where
$z$ is odd, for $j =2, 3, \cdots, r+1$; \\
$\bullet$  add the vertex $1+(p-1)(2r+1)$. \\
 When~$p$ is odd, \\
 \newline
 $\bullet$ from stream~$i$: use vertices~$i +
z(2r+1)$, where $z$ is even, for $i = 1, r+2, r+3, \cdots,
2r+1$;\\
$\bullet$ from stream~$j$: use vertices~$j + z(2r+1)$, where $z$ is
odd, for $j =2, 3,\cdots, r+1$;\\
$\bullet$  add vertices $i+(p-1)(2r+1)$  for $i=2, 3, \cdots, r$.

(2) If $1\le q\le r+1$,  then $r+2, r+3, \cdots, 2r+1,
q+1+(p-1)(2r+1), q+2+(p-1)(2r+1), \cdots, q+r+(p-1)(2r+1)\in D$,
which follows from condition (3) of Lemma \ref{lem10}. For $i \in \{
1, 2,\cdots, q\}$, the constraint stream $i$ is given as follows: $i
\vee i+(2r+1) \vee \cdots \vee i+p(2r+1) $. For $i\in \{q+1, \cdots,
2r+1\}$, the constraint stream $i$ is given as follows: $i\vee
i+(2r+1)\vee \cdots \vee i+(p-1)(2r+1)$.
To satisfy condition (4) of Lemma \ref{lem10}, there are four possible cases:\\
(2A)~$i\in D$ for some $i\in \{1,\cdots,q\}$ and $j+p(2r+1) \in D$ for some $j\in \{1,\cdots,q\}$;\\
(2B)~$i\in D$ for some $i\in \{1,  \cdots, q\}$ and $j+(p-1)(2r+1) \in D$ for some $j\in \{q+r+1,\cdots,2r+1\}$;\\
(2C)~$i\in D$ for some $i\in \{q+1,\cdots,r+1\}$ and $j+p(2r+1) \in D$ for some $j\in \{1,\cdots,q\}$;\\
(2D)~$i\in D$ for some $i\in \{q+1,\cdots,r+1\}$ and $j+(p-1)(2r+1)
\in D$ for some $j\in \{q+r+1,\cdots,2r+1\}$.

First consider the case (2A). We first discuss the situation $i\neq
j$. For stream $i$, we have already taken $i$ into $D$, satisfying
the first constrain in stream $i$, and hence we need take at least
$\lceil \frac{p-1}{2}\rceil$ vertices from stream $i$ to satisfy
remaining constraints. For stream $j$, we have already taken
$j+p(2r+1)$ into $D$, satisfying the last constrain in stream $j$,
and hence we need take at least $\lceil \frac{p-1}{2}\rceil$
vertices from stream $j$ to satisfy remaining constraints. For each
stream $t$ with $t\in \{1,\cdots,q\}\backslash \{i,j\}$,   we need
take at least
  $\lceil\frac{p}{2}\rceil$ vertices   into $D$ to satisfy its
constraint. For each stream $t$ with $t\in \{q+1,\cdots,r+1\}$, we
have already taken the vertex $t+(p-1)(2r+1)$ into $D$, satisfying
the last constrain in stream $t$, and there are $p-2$ remaining
constraints. Hence, we need take at least $\lceil
\frac{p-2}{2}\rceil$ vertices from stream $t$ to satisfy remaining
constraints. For each stream $t$ with $t\in \{r+2,\cdots,r+q\}$, we
have already taken $t$ and  $t+(p-1)(2r+1)$ into $D$, satisfying the
first and the last constrains in stream $t$, and there are $p-3$
remaining constraints. Hence, we need take at least $\lceil
\frac{p-3}{2}\rceil$ vertices from stream $t$ to satisfy remaining
constraints. For each stream $t$ with $t\in \{r+q+1,\cdots,2r+1\}$,
we have already taken $t$  into $D$, satisfying the first  in stream
$t$. Hence, we need take at least $\lceil \frac{p-2}{2}\rceil$
vertices  to satisfy remaining constraints. Therefore, we need at
least
$2r+2+(q-2)\lceil\frac{p}{2}\rceil+2\lceil\frac{p-1}{2}\rceil+(2r-2q+2)\lceil\frac{p-2}{2}\rceil
+(q-1)\lceil\frac{p-3}{2}\rceil$ vertices in all.

We now discuss the situation $i=j$. For stream $i$, we have already
taken $i$ and $i+p(2r+1)$ into $D$, satisfying the first and the
last constrains in stream $i$, and hence we need take at least
$\lceil \frac{p-2}{2}\rceil$ vertices from stream $i$ to satisfy
remaining constraints.  For each stream $t$ with $t\in
\{1,\cdots,q\}\backslash \{i\}$,  we need take at least
$\lceil\frac{p}{2}\rceil$ vertices   into $D$ to satisfy its
constraint. For each stream $t$ with $t\in \{q+1,\cdots,2r+1\}$, the
discussion is same as above. Therefore, we need at least
$2r+2+(q-1)\lceil\frac{p}{2}\rceil+(2r-2q+3)\lceil\frac{p-2}{2}\rceil
+(q-1)\lceil\frac{p-3}{2}\rceil$ vertices in all.



We now consider the case (2B). Similarly, we need at least
$\lceil\frac{p-1}{2}\rceil$ vertices to satisfy the remaining
constrains in stream $i$. For each stream $t$ with $t\in \{1,\cdots,
q\}\backslash \{i\}$, we need at lest $\lceil\frac{p}{2}\rceil$
vertices to satisfy the remaining constrains. For each stream $t$
with $t\in \{q+1,\cdots, r+1\}$, we need at lest
$\lceil\frac{p-2}{2}\rceil$ vertices to satisfy the remaining
constrains.  For each stream $t$ with $t\in \{r+2,\cdots, r+q\}$, we
need at least $\lceil\frac{p-3}{2}\rceil$ vertices to satisfy the
remaining constrains. For stream $j$, we need at least
$\lceil\frac{p-3}{2}\rceil$ vertices to satisfy the remaining
constrains. For each stream $t$ with $t\in \{q+r+1,\cdots,
2r+1\}\backslash \{j\}$, we need at least
$\lceil\frac{p-2}{2}\rceil$ vertices to satisfy the remaining
constrains. Therefore, we need at least
$2r+2+(q-1)\lceil\frac{p}{2}\rceil+\lceil\frac{p-1}{2}\rceil+(2r-2q+1)\lceil\frac{p-2}{2}\rceil
+q\lceil\frac{p-3}{2}\rceil$ vertices in all.

We now turn to  the case (2C). We need at least
$\lceil\frac{p-1}{2}\rceil$ vertices to satisfy the remaining
constrains in stream $j$. For each stream $t$ with $t\in \{1,\cdots,
q\}\backslash \{j\}$, we need at least $\lceil\frac{p}{2}\rceil$
vertices to satisfy its constrains. For stream $i$, we need at least
$\lceil\frac{p-3}{2}\rceil$ vertices to satisfy the remaining
constrains. For each stream $t$ with $t\in \{q+1,\cdots,
r+1\}\backslash \{i\}$, we need at least $\lceil\frac{p-2}{2}\rceil$
vertices to satisfy the remaining constrains. For each stream $t$
with $t\in \{r+2,\cdots, r+q\}$, we need at least
$\lceil\frac{p-3}{2}\rceil$ vertices to satisfy the remaining
constrains. For each stream $t$ with $t\in \{r+q+1,\cdots, 2r+1\}$,
we need at least $\lceil\frac{p-2}{2}\rceil$ vertices to satisfy the
remaining constrains. Therefore, we need at least
$2r+2+(q-1)\lceil\frac{p}{2}\rceil+\lceil\frac{p-1}{2}\rceil+(2r-2q+1)\lceil\frac{p-2}{2}\rceil
+q\lceil\frac{p-3}{2}\rceil$ vertices in all.

At last we consider the case (2D). For each stream $t$ with $t\in
\{1,\cdots, q\}$, we need at least $\lceil\frac{p}{2}\rceil$
vertices to satisfy its constrains. For stream $i$, we need at least
$\lceil\frac{p-3}{2}\rceil$ vertices to satisfy its remaining
constrains. For each stream $t$ with $t\in \{q+1,\cdots,
r+1\}\backslash \{i\}$, we need at least $\lceil\frac{p-2}{2}\rceil$
vertices to satisfy its remaining constrains.     For each stream
$t$ with $t\in \{r+2,\cdots, r+q\}$, we need at least
$\lceil\frac{p-3}{2}\rceil$ vertices to satisfy the remaining
constrains. For stream $j$, we need at least
$\lceil\frac{p-3}{2}\rceil$ vertices to satisfy its remaining
constrains. For each stream $t$ with $t\in \{r+q+1,\cdots,
2r+1\}\backslash \{j\}$, we need at least
$\lceil\frac{p-2}{2}\rceil$ vertices to satisfy the remaining
constrains. Therefore, we need at least
$2r+2+q\lceil\frac{p}{2}\rceil+(2r-2q)\lceil\frac{p-2}{2}\rceil
+(q+1)\lceil\frac{p-3}{2}\rceil$ vertices in all.

Finally, comparing the required minimum number of $D$ in all four
cases, we see that when $p$ is even, the minimum is
$\frac{(2r+1)p}{2}+q$, which is achieved in case (2A) under the
situation $i=j$, and when $p$ is odd, the minimum is
$\frac{(2r+1)(p-1)}{2}+2r+1$, which is achieved in cases (2C),(2D)
and  case (2A) under the situation  $i\neq j$.

Next, we construct an $r$-IC, achieving the bound as follows:\\
 When $p$ is even, \\
 \newline
 $\bullet$ from stream~$i$: use vertices $i + z(2r+1)$, where $z$ is even, for
$i = 1, r+2, r+3, \cdots, 2r+1$;\\
$\bullet$  from stream~$j$: use vertices $j + z(2r+1)$,~where
$z$ is odd, for $j =2, 3, \cdots, r+1$; \\
$\bullet$  add vertices $i+(p-1)(2r+1)$ for $i=r+2, r+3,\cdots, r+q$. \\
 When~$p$ is odd, \\
 \newline
 $\bullet$ from stream~$i$: use vertices~$i +
z(2r+1)$, where $z$ is odd, for $i = 1, 2, \cdots,q$;\\
$\bullet$ from stream~$j$: use vertices~$j + z(2r+1)$, where $z$ is
even, for $j=q+1, q+2,\cdots, 2r+1$.\\

 (3) The proof of  (3) is analogous. We
simply include the instruction for how to achieve an optimal set $D$ in this case.\\
When $p$ is even,\\
\newline
$\bullet$ from stream~$i$: use vertices~$i + z(2r+1)$,~where $z$ is
even, for $i= 1, r+2, r+3, \cdots, 2r+1$; \\
$\bullet$ from stream~$j$: use vertices~$j + z(2r+1)$,~where $z$ is
odd, for $j =2, 3, \cdots, r+1$;\\
$\bullet$  add vertices $i+(p-1)(2r+1)$ for $i=q+1, q+2, \cdots, 2r+1$ and $j+ p(2r+1)$ for $ j=2, 3,\cdots, q-r-1$.\\
When $p$ is odd,\\
\newline
$\bullet$   from stream $i$: use vertices~$i + z(2r+1)$,~where $z$
is odd, for $i=1, 2,\cdots, r$; \\
$\bullet$ from stream~$j$: use vertices~$j + z(2r+1)$,~where $z$ is
even, for $j= r+1, r+3, \cdots, 2r+1$.\\

\end{proof}
\section{$2$-locating dominating sets for cycle $C_n$}
Let $A$ and $B$ be two sets.  Define $A\triangle B$ as $(A- B)\cup
(B-A)$. For three vertices $x, u, v$, if $x\in D_r(u)\triangle
D_r(v)$, then we say $\{u,v\}$ are {\sl $r$-separated by $x$}, or
$x$ {\sl $r$-separates $\{u,v\}$}. Let $D$ be an $r$-LD for $C_n$.
Two different vertices $x$ and $y$ not in $D$ are {\sl
$D$-consecutive} if ether  $\{x+1, \cdots, y-1\}\subseteq D$ or
$\{y+1,\cdots,x-1\}\subseteq D$ holds. Note that a pair of
consecutive vertices $\{x,x+1\}$ not in $D$ are also
$D$-consecutive.


\begin{lemma}(\cite{BCO})\label{lem 12}
Let $r \geq 1$ be an integer. Suppose $D$ is an $r$-LD for $C_n$.
For every vertex $x$ in $D$, we have \\
(i) $x$ can $r$-separate at most two pairs of consecutive
vertices;\\
(ii) $x$ can $r$-separate at most two pairs of $D$-consecutive
vertices.
\end{lemma}

\begin{proof}
(i) For every vertex $x$ in $D$, it $r$-separates at most two pairs
of consecutive vertices $\{x-r-1, x-r\}$ and $\{x+r, x+r+1\}$.\\
(ii) Let $l$ and $l'$ be integers such that $0 < l \leq r$ and $l'
> r$.  $x$ can at most $r$-separate the following two types of
$D$-consecutive vertices: $(x\pm l,  x + l')$ and $(x\pm l,  x-l')$.
$\hfill\Box$
\end{proof}

\begin{lemma}\label{thm13}~(\cite{BCO})~
For $r \geq 2, ~n \geq 1$,~$M_2^{LD}(C_n) \geq
\lceil\frac{n}{3}\rceil$.\\
\end{lemma}
\begin{proof}
 Let $D$ be an $r$-LD of $C_n$. By Lemma \ref{lem 12}, and since there are $n-|D|$ pairs of $D$-consecutive
 vertices, we have $2|D|\ge n-|D|$.    $\hfill\Box$
\end{proof}

In here, we focus on  $r=2$.  Our main results  is the following
theorem.
\begin{theorem}\label{thm15} Let $C_n$ be a cycle with vertex set $\{x_1,\cdots x_n\}$.\\
(1) $M_2^{LD}(C_n)=n$ if $n\le 5$;\\
(2) $M_2^{LD}(C_n)=\lceil \frac{n}{3}\rceil+1$ if $n=6$ or
$n=6k+3$ ($k\ge 1$);\\
(3) $M_2^{LD}(C_n)=\lceil\frac{n}{3}\rceil$ if otherwise.

\end{theorem}

\begin{proof}
 When $n\le 5$, the distance of any two vertices in $C_n$ is no
more than $2$. Hence,  $M_2^{LD}(C_n)=n$. As a set with size two has
only three nonempty subsets, we know that $M_2^{LD}(C_6)\ge 3$. It
is easy to know that $D=\{x_1,x_3,x_5\}$ is a $2$-LD of $C_6$.
Therefore, $M_2^{LD}(C_6)= 3$.  In the following, we assume that
$n\ge 7$.

 $M_2^{LD}(C_n) \geq \lceil\frac{n}{3}\rceil$ holds by Lemma
\ref{thm13}, next we construct a $2$-LD achieving the
lower bound in the following cases:\\
$\bullet$ $n = 6k$, $D=\{x_i| i=6p+4, p\ge 0\}\cup\{x_i| i = 6q, q \geq 1\}$;\\
$\bullet$ $n = 6k+1$ or $6k+2$, $D=\{x_i| i = 6p+4, p\ge 0\}\cup \{x_i| i = 6q, q \geq 1\} \cup \{x_n\}$;\\
$\bullet$ $n = 6k+4$, $D=\{x_i| i = 6p+4, p\ge 0\}\cup\{x_i| i = 6q,
q \geq 1\}\cup \{x_{n-2}\}$;\\
$\bullet$ $n = 6k+5$ and $n>11$, $D=\{x_i| i = 6p+2, 0\le p\le k-2
\}\cup\{x_i| i = 6q, 1\le q \le k-1\}\cup \{x_{n-8},x_{n-7},x_{n-2},x_{n-1}\}$;\\
$\bullet$ $n =11$, $D=\{x_1,x_2,x_5,x_9\}$.\\

 Now we turn to the case $n= 6k+3$. By Lemma \ref{thm13}, we have known that  $M_2^{LD}(C_n) \geq 2k+1$.
 We first show that $M_2^{LD}(C_n) \geq 2k+2$.
Suppose to the contrary that  $D$ is a $2$-LD  for $C_n$ with $2k+1$
vertices. Then there are $4k+2$ pairs of $D$-consecutive vertices,
and hence every vertex in $D$ $2$-separates exactly two pairs of
$D$-consecutive vertices, and they are disjoint.  We have the following claims.\\
\\
{\bf Claim 1:} $D$ contains at most two consecutive vertices in $C_n$.\\
{\bf Proof of Claim 1:} Since each vertex in $D$ $2$-separates two
pairs of $D$-consecutive vertices, it follows that $D$ contains at
most four consecutive vertices in $C_n$. Suppose that $D$ contains
four consecutive vertices in $C_n$, without loss of generality, we
assume that $\{x_1,x_2,x_3,x_4\}\subseteq D$. Then both $x_1$ and
$x_4$ $2$-separate a pair of $D$-consecutive vertices $\{x_n,x_5\}$,
a contradiction. If $D$ contains three consecutive vertices in
$C_n$, without loss of generality, we assume that
$\{x_1,x_2,x_3\}\subseteq D$, then both $x_1$ and $x_3$ $2$-separate
a pair of $D$-consecutive vertices $\{x_n,x_4\}$, a contradiction.
$\Box$
\vskip 0.2cm
Assume that $D=\{x_{i_1}, x_{i_2},\cdots,
x_{i_{2k+1}}\}$ with
$i_1<i_2<\cdots<i_{2k+1}$.\\
\\
{\bf Claim 2:}  $|i_j-i_{j+1}|= 2$ or $4$ for all $j\in \{1,\cdots, 2k+1\}$.\\
 {\bf Proof of  Claim 2:}
Since $D_2(x)\neq \emptyset$ for any $x\not\in D$, it is easy to
know that $|i_j-i_{j+1}| \leq 5$. If $|i_j-i_{j+1}| = 5$ for some
$j\in\{1,\cdots,2k+1\}$, then $x_{i_j}, x_{i_{j+1}}$ both
$2$-separate the pair of consecutive vertices $(x_{i_j+2},
x_{i_j+3})$, a contradiction.

Suppose that $|i_j-i_{j+1}| = 1$ for some $j\in\{1,\cdots,2k+1\}$,
without loss of generality, we assume that $x_1\in D$ and $x_2\in
D$. By Claim 1, we know that $x_3\not\in D$ and $x_n\not\in D$. If
$x_4\in D$, then either $\{x_3,x_5\}$ or $\{x_3,x_6\}$ is a pair of
$D$-consecutive vertices. So, $x_1$ and $x_2$ $2$-separate the same
pair of $D$-consecutive vertices, a contradiction. Thus ~$x_4\not\in
D$. Similarly, $x_{n-1}\not\in D$. If  $x_{n-2}$ and $x_5$ are both
in $ D$, then
 they both $2$-separate the pair of $D$-consecutive vertices $\{x_n, x_3\}$, a contradiction.
  Without loss of generality, we take $x_5\not\in D$. $x_6\in D$ implies that
  the pair of $D$-consecutive vertices
  $\{x_3,x_4\}$ are $2$-separated by both  $x_1$ and $x_6$. It is a contradiction.
  $x_7\in D$ implies that the pair of $D$-consecutive vertices
  $\{x_4,x_5\}$ are $2$-separated by both $x_2$ and $x_7$. It is a contradiction.
  Hence, $x_6$ and $ x_7$ are both not in $D$. Thus, $D_2(x_5)= \emptyset$, a contradiction.
  Therefore, $|i_j-i_{j+1}|\neq 1$.

 Suppose that  $|i_j-i_{j+1}| = 3$ for some $j\in\{1,\cdots,2k+1\}$, without loss of generality, we
assume that $x_1\in D$ and $x_4\in D$. Then $x_n\in D$ or $x_5\in
D$,  which follows from the pair of $D$-consecutive vertices $\{x_2,
x_3\}$ require to be $2$-separated, however,  it contradicts with
$|i_j-i_{j+1}| \geq 2$.   \ \ $\Box$ \vskip 0.2cm



Since $C_n$ contains $6k+3$ vertices and there are $2k+1$ vertices
in $D$, thus by Claim 2,  there must exist some $j\in
\{1,\cdots,2k+1\}$  such that $|i_j-i_{j+1}| =|i_j-i_{j-1}|$.
However, if $|i_j-i_{j+1}| =|i_j-i_{j-1}|=2$, then $x_{i_{j-1}},
x_{i_{j+1}}$ both $2$-separate~$\{x_{i_j-1}, x_{i_j+1}\}$;
if~$|i_j-i_{j+1}| =|i_j-i_{j-1}|=4$, then there is no vertex in $D$
$2$-separate~$\{x_{i_j-1}, x_{i_j+1}\}$. Therefore, $M_2^{LD}(C_n)
\geq 2k+2$.

Now, we construct a $2$-LD for $C_n$ with $2k+2$ vertices as
follows: $D=\{x_i |i=6p+1$ or $ 6p+3, 0\le p\le k-1\}\cup\{x_{n-1},
x_{n-2}\}$. $\hfill\Box$

\end{proof}

\section{Conclusion}
The main purpose of this paper is to give the exact value of
$M_r^I(G)$ for paths and odd cycles for arbitrary positive integer
$r$, and of $M_2^{LD}(C_n)$. It would be of interest to extend the
latter to $r$-LDs for $r > 2$.
 \small {

}

\end{document}